\documentclass[12pt]{amsart}
\allowdisplaybreaks[1]
\usepackage{amsmath}
\usepackage{amssymb}
\usepackage{amsfonts}
\usepackage{verbatim}
\usepackage[usenames]{color}
\usepackage{float}
\usepackage{hyperref}
 \newtheorem{theorem}{Theorem}[section]
 
 \newtheorem{lemma}[theorem]{Lemma}
 \newtheorem{proposition}[theorem]{Proposition}

\theoremstyle{definition}

\newtheorem{observation}[theorem]{Observation}
\theoremstyle{remark}
\newtheorem{rem}[theorem]{Remark}
\newtheorem*{theorem*}{Theorem}
 \newtheorem*{corollary*}{Corollary}
\newtheorem{fact*}{Fact}

\newcommand\dd{\mathrm d}

\newcommand{\hilbert}{\mathcal{H}}

\newcommand{\K}{\mathcal{K}}

\newcommand{\Z}{\mathbb{Z}}
\newcommand{\T}{\mathbb{T}}

\newcommand{\D}{\mathbb{D}}
\newcommand{\C}{\mathbb{C}}

\newcommand{\R}{\mathbb{R}}

\newcommand{\cc}[1]{\overline{#1}}
\newcommand{\abs}[1]{\left\vert#1\right\vert}

\newcommand{\norm}[1]{\left\Vert#1\right\Vert}

\newcommand{\ip}[2]{\left\langle #1, #2 \right\rangle}
\newcommand{\ad}{^\ast}
\newcommand{\inv}{^{-1}}

\newcommand{\til}{\raise.17ex\hbox{$\scriptstyle\mathtt{\sim}$}}

\newcommand\beq{\begin{equation}}

\newcommand\eeq{\end{equation}}

\newcommand\red{\color{red}}

\newcommand\bbm{\begin{bmatrix}}
\newcommand\ebm{\end{bmatrix}}
\newcommand\bpm{\begin{pmatrix}}
\newcommand\epm{\end{pmatrix}}
\numberwithin{equation}{section}

\newlength{\Mheight}
\newlength{\cwidth}

\newcommand{\dfn}[1]{{\bf #1}\index{#1}}

\newcommand{\Hphi}{\hilbert_\phi}
\newcommand{\HKo}{\hilbert(K_1^{\max})}
\newcommand{\HKt}{\hilbert(K_2^{\min})}
\newcommand{\Sone}{S_1^{\max}}

\newcommand{\ktwo}{k_{2, w}^{\min}}
\newcommand{\kone}{k_{1, w}^{\max}}
\newcommand{\Kone}{K_1^{\max}}
\newcommand{\Ktwo}{K_2^{\min}}
\newcommand{\Mone}{M_{z_1}}
\newcommand{\Mtwo}{M_{z_2}}
\newcommand{\Madone}{M\ad_{z_1}}
\newcommand{\Madtwo}{M\ad_{z_2}}
\newcommand{\id}{1}
\newcommand{\TFR}{\text{TFR}}
\newcommand{\TFRs}{\text{TFRs}}

\title[Concrete Realizations and the MCC]{Analytic continuation of concrete realizations and the McCarthy Champagne conjecture}

\author{Kelly Bickel}
\author{J. E. Pascoe} 
\author{Ryan Tully-Doyle}
\date{\today}

\thanks{Bickel was supported in part by NSF-DMS Analysis grant \#2000088. Pascoe was supported in part by NSF-DMS Mathematical Science Postdoctoral Research Fellowship  \#1606260 and NSF-DMS Analysis grant 
\#1953963.}

\begin{document}

\begin{abstract}
In this paper, we give formulas that allow one to move between transfer function type realizations of multi-variate Schur, Herglotz and Pick functions, without adding additional singularities
except perhaps poles coming from the conformal transformation itself. In the two-variable commutative case, we use  a canonical  de Branges-Rovnyak model theory to
 obtain concrete realizations that analytically continue through the boundary for inner functions which are rational in one of the variables
(so-called \emph{quasi-rational functions}). We then establish a positive solution to McCarthy's Champagne conjecture for local to global matrix monotonicity in the settings of both two-variable quasi-rational functions and $d$-variable perspective functions. 
\end{abstract}
\maketitle
\tableofcontents 
\section{Introduction}
\subsection{Overview} 
Colloquially, realizations are ways of representing structured classes of functions using operators on a Hilbert space; these bridges between rich operator-theoretic results and concrete function theory have led to a myriad of important breakthroughs. Classic realizations-type formulae include the Nevanlinna representations for \textbf{Pick functions} (holomorphic functions mapping the upper half plane $\Pi$ to $\overline{\Pi}$) and the transfer function realizations for \textbf{Schur functions} (holomorphic functions mapping the unit disk $\mathbb{D}$ to $\overline{\mathbb{D}}$). 

In \cite{ag90}, J. Agler extended such one-variable formulae from systems engineering into functional analysis in several variables; this heralded in a period of rapid development for function theory on the bidisk $\mathbb{D}^2$ and polydisk $\mathbb{D}^d$, including extensions of Pick interpolation, the infinitesimal Schwarz lemma, L\"owner's theorem, and the Julia-Carath\'eodory theorem \cite{ag1, kn07, amy10a, amyloew, pascoelownote}. Realization theory has also been extended to noncommutative functional analysis, an area that has seen an explosion of activity in the last decade. Specifically,  J. Williams developed a realization theory in the free probability setting in \cite{will13}. In the free analysis setting, realizations for free Pick functions were developed in \cite{pptd16, ptd16}, which is part of a large body of recent and ongoing work in various noncommutative contexts \cite{bsv04, akv2006, bgm2005, knese08, pop89, po95, muhlysolel2004, muhlysolel2005, bbfh2009, bbf2007}. As in the commutative case, these realizations can be used to generalize classical theorems of complex analysis to functions of several noncommuting variables.
	
This paper investigates three foundational questions that one can ask about general realizations:
\begin{itemize}
\item[Q1:] When do the regularity properties of a realization exactly mimic those of the represented function?
\item[Q2:] How does one move between realization formulae without sacrificing fine behavior?
\item[Q3:] Are there settings where realizations possess identifiable concrete formulae?	
\end{itemize}	
In this paper, we use functional analysis on the bidisk to answer (Q1) and (Q3) for classes of two-variable Schur functions. We also develop a more general algebraic approach to (Q2), which yields a chain of operator expressions that relates Schur, Herglotz, and Pick-type structures and is applicable to the noncommutative setting. We then provide applications in the context of several variable functional analysis.

\subsection{Background} To motivate this investigation, consider the one-variable setting, and recall that Pick functions $f: \Pi \rightarrow \overline{\Pi}$
	can be written uniquely in the following form, called a \textbf{Nevanlinna representation},
		$$f(z) = a + bz + \int_{\mathbb{R}} \frac{1+tz}{t-z}\dd \mu(t)$$
	for some $a \in \mathbb{R}, b\in \mathbb{R}^{\geq 0},$ and $\mu$ a positive finite Borel measure on $\mathbb{R}$ \cite{nev22, lax02}.
	The complement of the support of $\mu$ is exactly the set where $f$ analytically continues to be real valued, and thus through the real line via
	the Schwarz reflection principle.  A similar fact holds for the earlier classical Herglotz integral representation for functions from the disk to the right half plane \cite{herg11, rie11}.  Nevanlinna and Herglotz functions have a number of applications, for example to the study of finite rank perturbations of self-adjoint operators; see the survey papers \cite{liawfrymark19, lt15}, book \cite{cmr06} and references within. 
	
Similarly, Schur functions $\phi: \mathbb{D} \rightarrow \overline{\mathbb{D}}$ possess a \textbf{transfer function realization} (or $\TFR$); i.e. they
	can be written in form\footnote{Here and throughout the paper, ``$1$'' denotes the identity operator on an appropriate Hilbert space that should be clear from the context. The notation ``$I$'' is reserved for an interval or open set.}
		$$\phi(z) = A+ B(1-zD)^{-1}zC \quad \text{ for } z \in \mathbb{D},$$
	where 
		$$U=\bbm A & B \\ C & D\ebm :  \bbm \mathbb{C} \\ \mathcal{M} \ebm   \rightarrow \bbm \mathbb{C} \\ \mathcal{M} \ebm  $$
	is a contraction on a Hilbert space $\mathbb{C} \oplus \mathcal{M}$, see \cite{helton74}. The operator $U$ can be chosen to be isometric, coisometric, or unitary; in each case, the choice is unique up to certain minimality assumptions and unitary equivalence, and there are concrete function theory interpretations for the canonical Hilbert spaces $\mathcal{M}$ and the operators $A, B, C, D$, see \cite{deb-rov1, balbol12, adrs97}. Under minimality assumptions, the set of $\tau \in \mathbb{T}$  where $1-\tau D$ is invertible is exactly the set where $\phi$ analytically continues with modulus $1$ and can therefore be analytically continued via the reflection principle on the disk.
	
The pioneering work of Agler in  \cite{ag90} (part of which was independently established by Kummert in \cite{kum89}) implies that each Schur function $\phi: \mathbb{D}^2 \rightarrow \overline{\mathbb{D}}$ has a two-variable $\TFR$
 and hence can be written as
		$$\phi(z) = A+ B(1-E_zD)^{-1}E_zC \quad \text{ for } z \in \mathbb{D}^2,$$
	where 
		$$U=\bbm A & B \\ C & D\ebm :  \bbm \mathbb{C} \\ \mathcal{M} \ebm   \rightarrow \bbm \mathbb{C} \\ \mathcal{M} \ebm  $$
	is a contraction on a Hilbert space $\mathbb{C} \oplus \mathcal{M}$ and can be chosen to be unitary, isometric, or coisometric. Here $\mathcal{M}$ decomposes as $\mathcal{M}_1 \oplus \mathcal{M}_2$ and  $E_z = z_1 P_1 + z_2 P_2$ where each $P_j$ is the projection onto $\mathcal{M}_j$.
	
	While Agler's initial proof was nonconstructive, influential work by Ball, Sadosky, and Vinnikov in \cite{bsv05}  used minimal scattering systems (for example, the so-called de Branges-Rovnyak model associated to $\phi$) and concrete Hilbert space geometry to produce and analyze more specific $\TFRs$.  
	They continued this seminal work with Kaliuzhnyi-Verbovetskyi  in \cite{bsv15}, which includes an exhaustive analysis of $\TFRs$ and connections between the geometric scattering structure and associated formal reproducing kernel Hilbert spaces. Ball and Bolotnikov conducted additional insightful work on canonical $\TFRs$ in \cite{balbol10, balbol12}. Many of these references also include results for the more general Schur-Agler class on $\mathbb{D}^d$, and we refer the reader to  \cite{gkvw08} for interesting related results concerning general Schur functions on $\mathbb{D}^d.$ 
	
If a Schur function $\phi$ is \textbf{inner}, i.e. if 
\[ \lim_{r \nearrow 1} | \phi(r \tau)| =1 \text{ for a.e. } \tau \in \mathbb{T}^2,\]
then the 
Hilbert space geometry from \cite{bsv05} simplifies dramatically. Indeed, in \cite{bk, bk2}, the first author and G. Knese  constructed particularly simple coisometric $\TFRs$ for two-variable inner functions using the Ball-Sadosky-Vinnikov machinery from \cite{bsv05}. This methodology yielded explicit formulae for the (reproducing kernel) Hilbert space $\mathcal{M}$, regularity properties of the functions in $\mathcal{M}$, and information about $A, B, C, D$. 
If $\phi$ is both rational and inner, then its $\TFRs$ come directly from sums of squares decompositions of related \textbf{stable} polynomials, i.e.~polynomials that do not vanish on $\mathbb{D}^2$, see \cite{bk, kum89, gewo04, knese08, w10}.  In this case, $\phi$ possesses a minimal $\TFR$ in the sense that if the degree of $\phi$ in $z_j$ is $m_j$, then 
$\mathcal{M}$ can be chosen so $\dim \mathcal{M} = m_1+m_2$. The proof of this minimality result is embedded in Kummert's work \cite{kum89}, and an extension with particularly clear exposition can be found in \cite {kn20}.  

The above minimality result was a key tool in  \cite{amyloew}. In this groundbreaking paper, Agler, McCarthy, and Young characterized multivariate \textbf{monotone matrix functions} via two types of monotonicity, a \emph{global} condition and a \emph{local} condition. Specifically, a real-valued function $f$ is \textbf{globally matrix monotone} on an open set $E \subseteq \mathbb{R}^d$ if for any positive integer $n$ and any pair of $d$-tuples of commuting $n\times n$ self-adjoint matrices $A=(A_1, \dots, A_d)$, $B=(B_1, \dots, B_d)$ with each $A_j \le B_j$  and joint spectrum in $E$, one has $f(A) \le f(B)$.  Meanwhile, $f$ is called \textbf{locally matrix monotone} on $E$ if the previous-described relation holds on positively-oriented paths in the variety of commuting self-adjoint matrices. As the exact notation of  local matrix monotonicity is 
cumbersome and not required in the current discussion, we refer the reader to \cite{amyloew, pascoelownote} for details. 

The work in \cite{amyloew} with later minor refinements in \cite{pascoelownote} yields  the following characterization of local matrix monotonicity:

\begin{theorem}\label{commlow} \cite{amyloew}. Let $E$ be an open set in $\mathbb{R}^d$. A function $f:\mathbb{E} \rightarrow \mathbb{R}$ is locally matrix monotone on $E$ if and only if $f$ analytically continues to $\Pi^d$ as a  map $f : E\cup \Pi^d \rightarrow \overline{\Pi}$ in the Pick-Agler class.
\end{theorem}
When $d = 1$ or $d = 2$, the Pick-Agler class is exactly the Pick class, i.e.~the set of analytic functions mapping $\Pi^d$ to $\overline{\Pi}$. More generally, the Pick-Agler class is the set of Pick functions that satisfy von Neumann's inequality after being converted to Schur functions via conformal mappings. The two-variable von Neumann inequality is known as And\^o's inequality \cite{and63} and fails in more than two variables \cite{par70, var74}. For $d >2$, this failure implies that the Pick-Agler class is a strict subset of the Pick class. For additional information about Pick-Agler functions and their structure, we refer the reader to \cite{amyloew}.

In the two-variable rational case, Agler, McCarthy, and Young  used the minimality of $\TFRs$ for rational inner functions  to characterize global matrix monotonicity on rectangles. 

\begin{theorem} \cite{amyloew}. If $E \subseteq \mathbb{R}^2$ is a rectangle and $f:\mathbb{E} \rightarrow \mathbb{R}$  is rational, then $f$ is globally matrix monotone on $E$ if and only if $f$ analytically continues to $\Pi^2$ as a  map $f : E\cup \Pi^2 \rightarrow \overline{\Pi}$ in the Pick class.\end{theorem}

The question of whether real-valued restrictions of Pick-Agler functions to convex sets in $\mathbb{R}^d$ are \emph{always} global matrix monotone functions has  colloquially become known as the McCarthy Champagne conjecture: 

\begin{center} \textbf{(MCC): Every $d$-variable Pick-Agler function that analytically continues across an open convex set $E \subseteq \mathbb{R}^d$ is globally matrix monotone when restricted to $E$. } \end{center} 
As the discussed further below, we establish the MCC in two important cases, giving compelling evidence for the overall validity of the conjecture.

\subsection{Summary of results} The bulk of this paper addresses the realization questions (Q1)-(Q3). In Section 
\ref{bkyourway}, we let $\phi$ be a two-variable inner function, review the particularly simple $\TFRs$  from  \cite{bk, bk2}, and further develop their properties. For example, in Theorem \ref{thm:AD}, we extend the analysis from \cite{bk2} to answer (Q3) and provide explicit formulae for each of $A,B,C,D$.

This allows us to address (Q1) for quasi-rational functions in Section \ref{quasirational}. Here, we say that a two-variable Schur function $\phi$ is \textbf{quasi-rational} with respect to an open $I \subseteq \mathbb{T}$ if $\phi$ is inner and extends continuously to $\T \times I$ with $\abs{\phi(\tau)} = 1$ for $\tau \in \T \times I$. The analysis from both Section \ref{bkyourway} and  \cite{bk, bk2} allows us to establish this key regularity property:

\begin{theorem*}\ref{invertonI}.
If $\phi$ is quasi-rational with respect to $I$ and $D$ is from Theorem \ref{thm:AD}, then $1 - E_\tau D$ is invertible for all $\tau \in \T \times I$.
\end{theorem*}

It is worth noting that this question of when operators of the form $(1-E_\tau D)$ are invertible is also connected to the study of robust stabilization in control engineering, see \cite{amyloew, dfcontrol}. 

Section  \ref{algebra} addresses (Q2) and shows how to move between realizations on different canonical domains while preserving delicate regularity behavior; see Theorems \ref{SchurToHerglotz} and \ref{thm:Nform}. Specifically, we show that on the level of algebra, the set of definition of a realization is the same as that when the domains have been conformally transformed, excepting obvious obstructions. In the noncommutative case, the results we obtain are completely clean, ``minimal" realization formulae that are canonical and therefore have maximal domain, similar to the results in \cite{ptdroyal}. Section  \ref{concretenp} contains an application of these theorems; we use the canonical realization from Theorem \ref{thm:AD} for  inner Schur functions on $\mathbb{D}^2$ to obtain canonical representations for real Pick functions on $\Pi^2.$ 

Section \ref{mccarthy} addresses our progress on the McCarthy Champagne conjecture. We first combine the machinery from Section \ref{algebra} with Theorem \ref{invertonI} to establish

\begin{theorem*}\ref{mccarthy1}. If $f$ arises from a two-variable quasi-rational function $\phi$, then the MCC holds for $f$. 
\end{theorem*}
For the exact details of the statement, including the domain where $f$ is globally matrix monotone as well as the connection between $f$ and its associated quasi-rational function $\phi$, see Section \ref{mccarthy}. In that section, we also study a class of $d$-variable  Pick-Agler functions known as  \textbf{commutative perspective functions}, which appear in the operator means literature \cite{andokubo80, effros09, ebadian2011, effroshansen13}. We show that the noncommutative L\"owner theorem from \cite{pastd13} implies that
\begin{theorem*}\ref{mccarthy2}.
		If $f$ is a $d$-variable commutative perspective function, then the MCC holds for $f$.
\end{theorem*}
 One surprising aspect of the precise statement of Theorem \ref{mccarthy2} is the following:  it only assumes that $f$ is locally matrix monotone on a positive cone $C \subseteq (0,\infty)^d$ but concludes that $f$ must actually be globally matrix monotone on all of $(0, \infty)^d$. 
\section{Two-variable realization formulae} \label{bkyourway}

We begin with the technical setup for the  de Branges-Rovnyak canonical model theory for two variable inner functions from \cite{bk, bk2}.  Throughout this section, let $\phi: \D^2 \to \D$ be a two variable inner function. 

Denote by $H^2 = H^2(\D^2)$ the Hardy space on the bidisk. First, we record some useful facts about the action of multiplication operators on $H^2$. For $j=1,2,$ let $M_{z_j}$ denote multiplication by $z_j$ in $H^2$ and recall that the adjoints are the backward shift operators defined by 
\[
 (M\ad_{z_2} f)(z) = \frac{f(z) - f(z_1, 0)}{z_2}, \hspace{.2in} (M\ad_{z_1} f)(z) = \frac{f(z) - f(0,z_2)}{z_1}
\]
for all $f \in H^2$ 
and so we have
\beq\label{eq1} 
f(z) = z_2(M\ad_{z_2}f)(z) + f(z_1, 0),
\eeq
\beq\label{eq2}
f(z) = z_1(M\ad_{z_1}f)(z) + f(0,z_2).
\eeq
Evaluating \eqref{eq2} at $z_2 = 0$ gives 
\beq\label{eq3}
f(z_1, 0) = z_1(\Madone f)(z_1, 0) + f(0,0),
\eeq
which can be plugged into \eqref{eq1} to produce the formula
\beq\label{fmult}
f(z) = z_2(\Madtwo f)(z) + z_1(\Madone f)(z_1, 0) + f(0,0).
\eeq

We now define the enveloping reproducing kernel Hilbert space for $\phi$ and the structured subspaces upon which the Agler model equation will be built. 
\begin{itemize}
\item Let $\K_\phi$ be the reproducing kernel Hilbert space
\[
\K_\phi = \hilbert\left[\frac{1 - \phi(z)\cc{\phi(w)}}{( 1- z_1 \cc{w_1})(1 - z_2 \cc{w_2})}\right] = H^2 \ominus \phi H^2;
\]
\item $S_1^{\max} =$ the maximum subspace of $\K_\phi$ invariant under $M_{z_1}$;
\item $S_2^{\min} = \K_\phi \ominus S_1^{\max}$;
\end{itemize}
Here $\mathcal{H}(K)$ denotes the Hilbert space of functions with reproducing kernel $K$. 
In \cite{bsv05}, Ball, Sadosky, and Vinnikov showed that with these definitions, $S_2^{\min}$ is invariant under $M_{z_2}$. We can then define these key Hilbert spaces:
\begin{itemize} 
\item $\HKo = S_1^{\max} \ominus z_1 S_1^{\max}$;
\item $\HKt = S_2^{\min} \ominus z_2 S_2^{\min}$,
\end{itemize}
where $z_j$ is shorthand for $M_{z_j}.$
As $\K_{\phi} = S_1^{\max} \oplus S_2^{\min}$, their reproducing kernels satisfy the question
\[ \frac{1 - \phi(z)\cc{\phi(w)}}{( 1- z_1 \cc{w_1})(1 - z_2 \cc{w_2})} = \frac{K_1^{\max}(z, w)}{1 -z_1 \cc{w_1}} + \frac{K_2^{\min}(z, w)}{1 -z_2 \cc{w_2}}.\] 
This immediately gives the associated \emph{Agler model equation}
\beq\label{modeleqn}
1 - \phi(z)\cc{\phi(w)} = (1 -z_1 \cc{w_1})K_2^{\min}(z, w) + (1 - z_2 \cc{w_2})K_1^{\max}(z, w).
\eeq
Set 
\[
\hilbert_\phi = \hilbert(K_2^{\min}) \oplus \hilbert(K_1^{\max}),
\]
so that each $f \in \hilbert_\phi$ can be written uniquely as $f= f_1 + f_2$ for $f_1 \in \HKt, f_2 \in \HKo$. Note that in contrast to Agler's original approach in \cite{ag90}, here we have explicit kernel structures to work with, which will allow direct calculations involving functions in $\hilbert_\phi$.

We now use \eqref{modeleqn} to derive a realization formula for $\phi$. Define an operator $V$ so that for all $w \in \mathbb{D}^2$,
\[
V\bbm 1 \\ \cc w_1 \ktwo \\ \cc{w}_2\kone \ebm \mapsto \bbm \ \cc{\phi(w)} \ \\ \ktwo \\ \kone \ebm,
\]
where $k_{1, w}^{\max} = K_1^{\max}(\cdot, w)$ and $k_{2, w}^{\min} = K_2^{\min}(\cdot, w)$.
Then the arguments in  \cite[pp. 6316-6318]{bk2} imply that $V$ extends to a unique isometry on $\C \oplus \HKt \oplus \HKo$. More generally, this type of argument is known as a lurking isometry argument, see  \cite{amy2020}, but often the underlying Hilbert space needs to be enlarged to guarantee that the resulting $V$ is isometric. Now, if we write
\beq \label{eqn:Udef}
V\ad = \bbm A & B \\ C & D \ebm : \bbm \C \\ \hilbert_\phi \ebm \mapsto \bbm \C \\ \hilbert_\phi \ebm,
\eeq
then for all $z\in \mathbb{D}^2$, 
\begin{equation} \label{eqn:AD1} 
\phi(z) = A + B(\id - E_z D)\inv E_z C,
\end{equation}
where $E_z = \Mone P_2 + \Mtwo P_1$ and $P_2, P_1 \in \mathcal{L}(\hilbert_\phi)$ are defined as follows: $P_2$ is the projection onto $\HKt$, and $P_1$ is the projection onto $\HKo$.

We can take advantage of the explicit structure of the model setup to derive concrete formulae for the blocks of $V\ad$.
\begin{theorem} \label{thm:AD} Let $\phi$ be a two-variable inner function with concrete realization \eqref{eqn:AD1}. 
Then the following formulas hold:
\begin{enumerate}
\item For all $x \in \C$, $A$ is given by
\[
Ax = \phi(0)x.
\]
\item For all $f \in \hilbert_\phi$, $B$ is given by 
\[
B\bbm f_1 \\ f_2 \ebm = (f_1+ f_2)(0).
\]
\item For all $x \in \C$, $C$ is given by
\[
Cx = \bbm P_2 \Madone \phi \\ P_1 \Madtwo \phi \ebm x.
\]
\item For all $f \in \hilbert_\phi$, $D$ is given by 
\[
D\bbm f_1 \\ f_2 \ebm = \bbm P_2 \Madone (f_1 + f_2) \\ P_1 \Madtwo (f_1 + f_2) \ebm.
\]
\end{enumerate}
\end{theorem}

\begin{proof}
The formulas for $A$ and $B$ are given in \cite[Remark $5.6$]{bk}. The formulas for $C$ and $D$ are proved in Lemmas \ref{Dformula} and  \ref{Cformula} below. \end{proof}

\begin{rem} 
The more general class of so-called \emph{weakly} coisometric realizations for $d$-variable Schur-Agler functions (Schur functions that also satisfy von Neumann's inequality) and their associated $A, B, C, D$ formulas were studied earlier in \cite{balbol12}. Specifically, Definition $3.1$ and Theorem $3.4$ in \cite{balbol12} also establish the formulas for $A$ and $B$ above and imply that $C$ and $D$ must each satisfy a so-called structured Gleason problem.

The concrete function theory interpretations for $C$ and $D$ in Theorem \ref{thm:AD} are also related to the technical and extensive work in \cite{bsv15}. In particular, in Theorem $5.9$, the authors assume that a given $d$-variable Schur-Agler function $\varphi$ possesses a so-called \emph{minimal augmented Agler decomposition} and use it to construct a specific unitary realization for $\varphi$ via the theory of scattering systems and formal reproducing kernel Hilbert spaces. The $A, B, C, D$ formulas that they obtain are quite similar to those of the cosimetric realization in Theorem \ref{thm:AD} above. 
\end{rem}

The following lemma will simplify later computations. Part of it appears as Proposition $3.5$ in \cite{bag}, but we include the simple proof here for completeness.
\begin{lemma} \label{lem:bshift} Let $\phi$ be a two-variable inner function with associated Hilbert spaces defined as above. Then $M_{z_2}^* \phi \in S^{\max}_1$ and $M_{z_1}^* \phi \in S^{\min}_2$. Furthermore
$ M_{z_2}^* \hilbert_\phi \subseteq S^{\max}_1$ and $ M_{z_1}^* \hilbert_\phi  \subseteq S^{\min}_2.$
\end{lemma}
\begin{proof} As $S^{\min}_2$ is invariant under $M_{z_2}$, it follows easily that $S^{\max}_1$ is invariant under $M^*_{z_2}$. Thus,  $ M_{z_2}^* \hilbert(K_1^{\max}) \subseteq S^{\max}_1$.
Now rewrite the model equation \eqref{modeleqn} as the following equality of positive kernels:
\[ \frac{1}{ 1- z_1 \cc{w_1}} + z_2 \cc{w_2}\frac{K_1^{\max}(z, w)}{1 -z_1 \cc{w_1}} =  \frac{\phi(z)\cc{\phi(w)}}{1- z_1 \cc{w_1}} +  K_2^{\min}(z, w) + \frac{K_1^{\max}(z, w)}{1 -z_1 \cc{w_1}}.\]
This shows that $\phi$ and each $f \in \hilbert( K_2^{\min})$ can be written as $g(z_1) + z_2 h(z)$ where $g \in H^2(\mathbb{D})$ and $h \in S^{\max}_1$. Then the definition of $M_{z_2}^*$ immediately implies $M_{z_2}^* \phi \in S^{\max}_1$ and $ M_{z_2}^* \hilbert(K_2^{\min}) \subseteq S^{\max}_1,$ which establishes the $S^{\max}_1$ inclusions.
The $S_2^{\min}$ inclusions follow from an analogous argument.
\end{proof}

We can now establish the formulae for $C$ and $D$. 

\begin{lemma}\label{Dformula}Let $\phi$ be a two-variable inner function with concrete realization \eqref{eqn:AD1}.  Then 
\[ \text{ for all } f = \bbm f_1 \\ f_2 \ebm \in \Hphi, \ \ 
D\bbm f_1 \\ f_2 \ebm  = \bbm P_2 \Madone (f_1 + f_2) \\ P_1 \Madtwo (f_1 + f_2) \ebm.
\]
\end{lemma}
\begin{proof}
Make the decomposition $Df = \bbm (Df)_1 \\ (Df)_2 \ebm$. We first establish the formula for $(Df)_2$ and then consider $(Df)_1.$

 By \cite[Remark 5.6]{bk2}, $(Df)_2$ is the unique function in $\HKo$ satisfying
\[
(Df)(0,z_2) = \frac{(f_1+f_2)(0,z_2) - (f_1+f_2)(0)}{z_2} = M\ad_{z_2}(f_1 + f_2)(0,z_2).
\]
By Lemma \ref{lem:bshift}, we have $M\ad_{z_2}(f_1+f_2) \in \Sone$.
Thus, we can write 
\[
M\ad_{z_2}(f_1 + f_2) = P_1 M\ad_{z_2}(f_1 + f_2) + (1-P_1) M\ad_{z_2}(f_1 + f_2),
\]
where $(1-P_1)$ projects $\Sone$ onto $z_1\Sone$. As $(1-P_1) M\ad_{z_2}(f_1 + f_2)$ is thus divisible by $z_1$, we have
\[
M\ad_{z_2}(f_1 + f_2)(0,z_2) = P_1 M\ad_{z_2}(f_1 + f_2)(0,z_2)
\]
and so by uniqueness, $(Df)_2 = P_1 M\ad_{z_2}(f_1 + f_2).$

To establish the formula for $(Df)_1$, note that by \cite[Remark 5.6]{bk2}, $(Df)_1$ is the unique function in $\HKt$ satisfying 
\beq \label{eqn:Dform}
 (Df)_1(z) = \frac{(f_1+f_2)(z) -(f_1+f_2)(0) - z_2 (Df)_2(z)}{z_1}.
\eeq
By the formula in \eqref{fmult}, we have
\[
(f_1 + f_2)(z) = z_2\Madtwo (f_1 + f_2)(z) + z_1\Madone (f_1 + f_2)(z_1, 0) + (f_1+f_2)(0),
\]
and so 
\begin{align*}
(Df)_1(z) &= \frac{f_1(z) + f_2(z) - f_1(0) - f_2(0) - z_2 (Df)_2(z)}{z_1} \\
&= \frac{z_2 M\ad_{z_2} (f_1 + f_2)(z) + z_1 M\ad_{z_1}(f_1 + f_2)(z_1, 0)- z_2 P_1 M\ad_{z_2}(f_1 + f_2)(z)}{z_1} \\
&= M\ad_{z_1}(f_1 + f_2)(z_1, 0) + \frac{z_2 (1 - P_1) M\ad_{z_2}(f_1 + f_2)(z)}{z_1} \\
&= M\ad_{z_1}(f_1 + f_2)(z_1, 0) + z_2\Madone( 1 - P_1) M\ad_{z_2}(f_1 + f_2)(z) \\
&= \Madone\left[(f_1 + f_2)(z_1, 0) + z_2(1-P_1)\Madtwo(f_1 + f_2)(z)\right] \\
&= \Madone\left[(f_1 + f_2)(z_1, 0) + z_2\Madtwo(f_1 + f_2)(z) - z_2 P_1 \Madtwo(f_1 + f_2)(z)\right] \\
&= \Madone\left[(f_1 + f_2)(z) - z_2 P_1 \Madtwo (f_1 + f_2)(z)\right] \hspace{.1in} \text{ (by \eqref{eq1})}\\
&= \Madone(1 - z_2P_1 \Madtwo)(f_1 + f_2)(z),
\end{align*}
where we again used the fact that $(1-P_1)M\ad_{z_2}(f_1 + f_2)$ is divisible by $z_1$. So, we have
\begin{align*}
(Df)_1 &= \Madone(1 - M_{z_2} P_1 \Madtwo)(f_1 + f_2)\\
& = P_2 \Madone (f_1+f_2) + P_2 M_{z_2} \Madone P_1 \Madtwo(f_1 + f_2),
\end{align*}
since $(Df)_1 \in \hilbert(K^{\min}_2)$. By Lemma \ref{lem:bshift}, $ \Madone P_1 \Madtwo(f_1 + f_2) \in S^{\min}_2$. Thus, 
$M_{z_2} \Madone P_1 \Madtwo(f_1 + f_2) \in z_2 S^{\min}_2$ and so, is annihilated by $P_2$. This implies
$
(Df)_1 = P_2\Madone (f_1 + f_2),
$
and establishes the claim.
\end{proof}

\begin{lemma}\label{Cformula} Let $\phi$ be a two-variable inner function with concrete realization \eqref{eqn:AD1}. 
For all $x \in \C$, $C$ is given by
\[
Cx = \bbm P_2 \Madone \phi \\ P_1 \Madtwo \phi \ebm x.
\]
\end{lemma}
\begin{proof}
The proof is similar to the argument for $D$ in Lemma \ref{Dformula}, so we give a sketch of the idea but omit some of the finer details. By linearity, we can let $x=1$. By \cite[Remark 5.6]{bk2}, $(C1)_2$ is the unique function in $\HKo$ with
\[
(C1)_2(0,z_2) = (\Madtwo \phi)(0,z_2).
\]
Lemma \ref{lem:bshift}, $M_{z_2}^*\phi \in \mathcal{S}^{\max}_1$ and then the same rationale as in Lemma \ref{Dformula} implies that $(C1)_2 = P_1 M_{z_2}^* \phi.$ To handle $(C1)_1$, write 
\[ \phi(z) = z_2 M_{z_2}^* \phi(z) + z_1 M_{z_1}^*\phi (z_1,0) + \phi(0).\]
Then  \cite[Remark 5.6]{bk2} implies that $(C1)_1$ satisfies
\begin{equation} \label{eqn:C1} z_1 (C1)_1(z) +z_2(C1)_2(z) = \phi(z) -\phi(0).\end{equation}
Substituting the formulas for $\phi$ and $(C1)_2$ into \eqref{eqn:C1} and solving for $(C1)_1$ yields
\[ 
\begin{aligned}
(C1)_1(z) &= M_{z_1}^*\phi(z_1,0) + \frac{ z_2 M^*_{z_2}\phi(z)-z_2 P_1 M_{z_2}^*\phi(z)}{z_1} \\
& =  M_{z_1}^*\phi(z_1,0) +  M^*_{z_1} \left(z_2 M^*_{z_2}\phi-z_2 P_1 M_{z_2}^*\phi \right)(z) \\ 
& = M^*_{z_1}\phi(z) - z_2 M^*_{z_1} P_1 M_{z_2}^*\phi(z),
\end{aligned}
\]
where we used the fact that $M_{z_2}^* \phi \in \mathcal{S}^{\max}_1.$ As $(C1)_1 \in \hilbert(K^{\min}_2)$, 
\[ (C1)_1 = P_2M^*_{z_1}\phi - P_2 M_{z_2} M^*_{z_1} P_1 M_{z_2}^*\phi.\]
Then Lemma \ref{lem:bshift} implies that $M^*_{z_1} P_1 M_{z_2}^*\phi \in \mathcal{S}_2^{\min}$ and so, $M_{z_2} M^*_{z_1} P_1 M_{z_2}^*\phi$ is annihilated by $P_2.$  This implies $(C1)_1 = P_2 \Madone \phi$ and completes the proof.
\end{proof}

We now show that the operator $D$ exhibits additional behavior resembling that of the backward shift $M\ad_{z_i}$.

\begin{proposition}\label{Dshift} Let $\phi$ be a two-variable inner function with concrete realization \eqref{eqn:AD1}.  Then for all $w \in \mathbb{D}^2$
\[
D \bbm \ktwo \\ \kone \ebm = \bbm \cc{w_1} \ktwo \\ \cc{w_2}\kone \ebm - \cc{\phi(w)}F, 
\ \ \text{ where } F = \bbm P_2 M\ad_{z_1} \phi \\ P_1 M\ad_{z_2}\phi \ebm.
\]
\end{proposition}
\begin{proof}
By Lemma \ref{lem:bshift},  $M\ad_{z_2} (\kone + \ktwo) \in \Sone$. Rearranging the model equation \eqref{modeleqn} and applying the operator $P_1\Madtwo$ to each side gives
\begin{align}
P_1 \Madtwo [\kone + \ktwo] &= P_1[ z_1 \cc{w_1}\Madtwo\ktwo + \cc{w_2}\kone - \cc{\phi(w)} \Madtwo\phi] \notag \\
&= \cc{w_2}\kone - \cc{\phi(w)}P_1 M\ad_{z_2}\phi, \label{Dkone}
\end{align}
since $z_1 \Madtwo \ktwo \in z_1 S^{\max}_1$ and hence, is annihilated by $P_1$. Similarly,
\begin{align}
P_2 \Madone [\kone + \ktwo] &= P_2[  \cc{w_1}\ktwo + z_2\cc{w_2} \Madone \kone - \cc{\phi(w)} \Madone\phi] \notag \\
&= \cc{w_1}\ktwo - \cc{\phi(w)}P_2 M\ad_{z_1}\phi,  \label{Dktwo} 
\end{align}
since $z_2 \Madone \kone  \in z_2 S^{\min}_2$.
Now applying Lemma \ref{Dformula} to $D \bbm \ktwo \\ \kone \ebm$ and using the expressions in \eqref{Dkone}, \eqref{Dktwo} gives the desired formula.
\end{proof}

\section{Boundary behavior of quasi-rational functions} \label{quasirational}
As in the last section, let $\phi$ be a two-variable inner function on $\mathbb{D}^2$. 
We now examine the behavior of the concrete realization of $\phi$ from \eqref{eqn:AD1} at points on the distinguished boundary $\mathbb{T}^2$. The goal is to show that if $\phi$ extends continuously at part of the boundary, then so does the realization. Equivalently, we want to show that the operator $1 - E_\tau D$ is invertible on some open set of boundary points where $\phi$ is well behaved.  

This problem is generally intractable via current methods, so we restrict to a special class of inner functions. Specifically, we say that a Schur function $\phi$ is \textbf{quasi-rational} with respect to an open $I \subseteq \mathbb{T}$ if $\phi$ is inner and extends continuously to $\T \times I$ with $\abs{\phi(\tau)} = 1$ for $\tau \in \T \times I$.  To get a sense of the definition, recall that every one-variable inner function that extends continuously to $\mathbb{T}$ is a finite Blaschke product. Thus, if $\phi$ is quasi-rational, then for each $\tau_2 \in I$, the one-variable function $\phi(\cdot, \tau_2)$ must be a finite Blaschke product.

\begin{rem} Before proceeding further, one should note that  the set of quasi-rational functions is quite large. To generate  examples, let $p\in \mathbb{C}[z_1, z_2]$ be a polynomial of degree $(m_1, m_2)$ that does not vanish on $\mathbb{D}^2$ and let $\psi = \frac{\tilde{p}}{p}$, where $\tilde{p}(z)= z_1^{m_1} z_2^{m_2} p(1/\bar{z}_1, 1/\bar{z}_2)$. Without loss of generality, we can assume that $p, \tilde{p}$ have no common factors. Then $\psi$ is a rational inner function (and all rational inner functions have this form, see \cite{rud69}) and we can define the set
\[ J_\psi =\left \{ \tau_2 \in \mathbb{T}: \text{ there exists } \tau_1 \in \mathbb{T} \text{ with } p(\tau_1, \tau_2) =0\right\},\]
which contains at most $m_1 m_2$ points.
Now let $\theta$ be any one-variable inner function that extends continuously to an open set $J \subseteq \mathbb{T}$. 
Let $I \subseteq J$ be any open set such that $\theta(I) \subseteq \mathbb{T} \setminus J_\psi$. Then the two-variable function $\phi$ defined by
\[ \phi(z) = \psi\left( z_1, \varphi(z_2) \right),\]
is quasi-rational with respect to $I$. Furthermore, it is immediate that the set of quasi-rational functions with respect to $I$ is closed under finite products.

The class of quasi-rational functions with respect to $I$ is also closed in a stronger sense. Specifically, assume that $(\phi_n)$ is a sequence of quasi-rational functions on $I$ that converges to some function $\phi$ both in the $H^2(\mathbb{D}^2)$ norm and locally uniformly on $\mathbb{D}^2 \cup (\mathbb{T} \times I)$. The first condition implies that the limit function $\phi$ is inner and the second condition implies that $\phi$ extends continuously to $\mathbb{T} \times I$. Thus, $\phi$ is quasi-rational with respect to $I$.
\end{rem}

Then for quasi-rational functions, we prove the following result:

\begin{theorem}\label{invertonI} 
Let $\phi$ be quasi-rational with respect to an open $I\subseteq \mathbb{T}$. Then in the concrete realization \eqref{eqn:AD1}, the operator $1 - E_\tau D$ is invertible for all $\tau \in \T \times I$.
\end{theorem}

Before proving the theorem, we prepare some model machinery. 
\begin{observation}\label{boundaryextension}
  By Theorem 1.5 in \cite{bk}, there is an open set $\Omega$ containing $(\overline{\D} \times I)\cup(\T \times \D) \cup \D^2$ on which $\phi$ and all functions in $\HKo$ and $\HKt$ extend to be analytic. Furthermore, point evaluation in $\Omega$ is bounded on these spaces, and the kernels $\Kone, \Ktwo$ extend to be sesquianalytic on $\Omega \times \Omega$. 
\end{observation}

Observe that for $\tau_2 \in I$, the one-variable inner function $\phi_{\tau_2} = \phi(\cdot, \tau_2)$ is well defined and possesses an associated one-variable reproducing kernel Hilbert space defined by
\[
 \mathcal{K}_{\phi_{\tau_2}} = \mathcal{H}\left[ \frac{1 - \phi_{\tau_2}(z_1)\cc{\phi_{\tau_2}(w_1)}}{1 - z_1 \cc{w_1}}\right],
\]
which is a subspace of the one-variable Hardy space $H^2(\mathbb{D})$. 
We connect these $ \mathcal{K}_{\phi_{\tau_2}}$ to the subspaces associated to our realizations via the following lemma.

\begin{lemma}\label{unitarylemma} Let $\phi$ be quasi-rational with respect to an open $I\subseteq \mathbb{T}$. 
 Then the map $J_{\tau_2}: \HKt\to \mathcal K_{\phi_{\tau_2}}$ defined by
$
  J_{\tau_2}f = f(\cdot, \tau_2)
 $
 is unitary for all $\tau_2 \in I$.
\end{lemma}
\begin{proof} 
 This proof uses ideas from the proofs of \cite[Theorem 1.6]{bk} and \cite[Theorem 2.2]{bg17}.  For this proof, one should recall that $\HKt$ is a subspace of $H^2(\mathbb{D}^2)$ and $\K_{\phi_{\tau_2}}$ is a subspace of $H^2(\mathbb{D})$. 

 We first claim that for $\tau_2 \in I$, the restriction map $J_{\tau_2}: \HKt\to H^2(\mathbb{D})$ preserves inner products.   
Fix functions $f, g \in \HKt$. Then for almost every $z_2 \in \mathbb{T}$, $f(\cdot, z_2), g(\cdot, z_2) \in L^2(\mathbb{T})$ and we can define
 \beq\label{Ffg}
  F_{f,g}(z_2) =  \ip{f(\cdot, z_2)}{g(\cdot, z_2)}_{L^2(\T)}. \eeq
Let $\sigma$ denote normalized Lebesgue measure. Then by H\"older's inequality, we obtain
\[ \int_{\T} |  F_{f,g}(z_2) | d \sigma(z_2) \le \int_{\T} \| f(\cdot, z_2) \|_{L^2(\mathbb{T})}  \| g(\cdot, z_2) \|_{L^2(\mathbb{T})} d \sigma(z_2) \le \|f \|_{H^2} \|g\|_{H^2}, \]
which implies $F_{f,g} \in L^1(\T)$. Since $f, g \in \HKt$ and $\HKt \perp_{H^2} z_2 \HKt$, we have
 \[
  f \perp_{L^2} z_2^j g \text{ for all } j \in \Z/\{0\}.
 \]
 Then the Fourier coefficients of $F_{f,g}$ for $j \in \Z/\{0\}$ are given by 
 \[
  \widehat{F_{f,g}}(-j) = \int_{\T} z_2^j F_{f,g}(z_2) \, d\sigma(z_2) = \int_{\T^2} z_2^j f(z) \cc{g(z)} \, d\sigma(z) = 0,
 \]
 and so it is straightforward that
 \[
  F_{f,g}(z_2) = \widehat{F_{f,g}}(0) = \ip{f}{g}_{\HKt} \text{ for a.e. } z_2 \in \T.
 \]
 By Observation  \ref{boundaryextension}, $f$ and $g$ are analytic on an open set $\Omega$ containing $\overline{\D} \times I$, which implies both that for every $\tau_2 \in I$, $f(\cdot, \tau_2), g(\cdot, \tau_2) \in H^2(\mathbb{D})$ and
 the  formula for $F_{f,g}$ in \eqref{Ffg} is well defined and continuous on $I$. This immediately gives
 \[
  \ip{f(\cdot, \tau_2)}{g(\cdot, \tau_2)}_{H^2(\mathbb{D})} = F_{f,g}(\tau_2) = \widehat{F_{f,g}}(0) = \ip{f}{g}_{\HKt},
 \]
so the restriction map preserves inner products for each $\tau_2 \in I$.
 
 To finish the proof, we need to show $J_{\tau_2}$ maps onto $\K_{\phi_{\tau_2}}$. By Observation \ref{boundaryextension}, for any $\tau_2 \in I$, we can let $z_2, w_2 \to \tau_2$ in the model equation \eqref{modeleqn} to obtain
\[
  \frac{1 - \phi_{\tau_2}(z_1)\cc{\phi_{\tau_2}(w_1)}}{1 - z_1 \cc{w_1}} = K_2^{\min}((z_1, \tau_2), (w_1, \tau_2)) = J_{\tau_2}\left(  k^{\min}_{2,(w_1, \tau_2)} \right) (z_1).
\]
 To show the range of $J_{\tau}$ is in $\K_{\phi_{\tau_2}}$, assume that $f \in \hilbert(K^{\min}_2)$ and $J_{\tau_2} f \perp 
  \K_{\phi_{\tau_2}}$ in $H^2(\mathbb{D})$. Since $J_{\tau_2}$ preserves inner products, this implies that for all $w_1 \in \mathbb{D}$
  \[ 0 =   \ip{f(\cdot, \tau_2)}{ k^{\min}_{2,(w_1, \tau_2)} (\cdot, \tau_2)}_{H^2(\mathbb{D})} =  \ip{f}{k^{\min}_{2,(w_1, \tau_2)}}_{\hilbert(K^{\min}_2)}  =f(w_1, \tau_2).\]
  Since $J_{\tau_2}$ preserves norms, $f \equiv 0$ and so $J_{\tau_2}$  maps into  $\K_{\phi_{\tau_2}}$.
  
Finally, as  $J_{\tau_2}$  preserves norms and its range contains all of the reproducing kernel functions of $\K_{\phi_{\tau_2}}$, $J_{\tau_2}$ must be surjective. 
\end{proof}

We are now ready to prove the main theorem.

\begin{proof}[Proof of Theorem \ref{invertonI}] 
Fix $(\tau_1, \tau_2) \in \mathbb{T} \times I.$ The proof has two parts. We start by showing that the operator $(1 - E_\tau D)$ has dense range and then we will show $(1 - E_\tau D)$ is bounded below.

First, proceed by contradiction and assume that $(1 - E_\tau D)$ does not have dense range. Recall that Proposition \ref{Dshift} implies
\beq\label{range}
 (\id - E_\tau D) \bbm \ktwo \\ \kone \ebm = \bbm (1 - \tau_1\cc{w_1})\ktwo \\ (1 - \tau_2\cc{w_2})\kone \ebm + \cc{\phi(w)}E_\tau F,
\eeq
where $F$ is defined in Proposition \ref{Dshift}. Then there must exist some non-trivial $g \in \Hphi$ orthogonal to all functions with the form given in \eqref{range}. Writing 
\[ g = \bbm g_1 \\ g_2 \ebm \ \text{ for $g_1 \in \mathcal{H}(K^{\min}_2),$ and $g_2 \in  \mathcal{H}(K^{\max}_1),$} \]
we can compute
\[
\begin{aligned}
 0 &= \ip{g}{(1 - E_\tau D) \bbm \ktwo \\ \kone \ebm}_{\Hphi} \\
 &= (1 - \cc{\tau_1}w_1)g_1(w) + (1 - \cc{\tau_2}w_2)g_2(w) + \phi(w)\ip{g}{E_\tau F}_{\Hphi},
 \end{aligned}
\]
for all $w\in \mathbb{D}^2$. By Observation \ref{boundaryextension}, we can take $w \to \tau$  and reduce this to
\[
 0 = \phi(\tau) \ip{g}{E_\tau F}_{\Hphi}
\]
and so  $\ip{g}{E_\tau F}_{\Hphi} = 0$. Then 
\beq\label{gdance}
 (1 - \cc{\tau_1} w_1)g_1(w) = - (1 - \cc{\tau_2}w_2)g_2(w) 
 \eeq
and in particular, by Observation \ref{boundaryextension}, we can take limits to points in $\mathbb{D} \times I$ to conclude
$
 g_1(z_1, \tau_2) = 0$ for all $z_1 \in \D.$
Since $g_1 \in \HKt$, an application of Lemma \ref{unitarylemma} implies that
\[
 \norm{g_1}_{\HKt} = \norm{g_1(\cdot, \tau_2)}_{\mathcal{K}_{\phi_{\tau_2}}} = 0.
\]
Thus $g_1 \equiv 0$, and by \eqref{gdance} we also have $g_2 \equiv 0$. Then $g \equiv 0$, which is a contradiction. We conclude that $(1 - E_\tau D)$ has dense range.

Now we show that $(1 - E_\tau D)$ is bounded below. Proceeding by contradiction, assume that $(1 - E_\tau D)$ is not bounded below. Then there is a sequence of functions $\{g^n\} \subset \Hphi$ such that $\norm{g^n}_{\Hphi} = 1$ and $\lim_{n\to \infty} \norm{(1 - E_\tau D) g^n}_{\Hphi} = 0$. 

Let $\tilde{I}$ be a closed interval in $I$ containing $\tau_2$ in its interior. By Observation \ref{boundaryextension} and the uniform boundedness principle, point evaluation in $\Hphi$ is uniformly bounded on $K:=\overline{\D} \times \tilde{I}$. That is, there is a $C >0$ so that
\begin{equation} \label{eqn:Cestimate}
 \abs{f(z)} \leq C \norm{f}_{\Hphi} \text{ for all } z \in K, f \in \Hphi.
\end{equation}
Furthermore, for $i=1,2$, this implies that for $z \in K$,
\[  | z_i (D g^n)_i(z) -\overline{\tau_i} z_i g_i^n(z) | 
=| z_i(1-E_{\tau}Dg^n)_i(z) | \le  C\|   (1-E_{\tau}Dg^n) \|_{\Hphi}.\] 
By \eqref{eqn:Dform}, we can conclude that 
\[ 
\begin{aligned}
 | g_1^n(z) &- \cc{\tau_1}z_1 g_1^n(z) + g_2^n(z) - \cc{\tau_2}z_2 g_2^n(z) - g_1^n(0) - g_2^n(0)| \\
&=  | z_1 (D g^n)_1(z) -\overline{\tau_1} z_1 g_1^n(z) + z_2 (D g^n)_2(z) -\overline{\tau_2} z_2 g_2^n(z)|  \\
&\le 2C \| (1-E_{\tau} D)g^n \|_{\Hphi}  \quad \text{ for all } z \in K.
\end{aligned}
\]
Setting $z=\tau$, we have $|g_1^n(0) + g_2^n(0)| \le 2C \| (1-E_{\tau} D)g^n \|_{\Hphi} $ and so 
\[ | (1- \cc{\tau_1}z_1) g_1^n(z) +(1- \cc{\tau_2}z_2) g_2^n(z) | \le 4C \| (1-E_{\tau} D)g^n \|_{\Hphi}, \]
for all $z \in K.$ In particular, setting $z_2 = \tau_2$ gives
\begin{equation} \label{eqn:g1estimate} | (1- \cc{\tau_1}z_1) g_1^n(z_1, \tau_2)| \le  4C \| (1-E_{\tau} D)g^n \|_{\Hphi},\end{equation}
for all $z_1 \in \overline{\D}$.  We can use this to deduce that $\| g_1^n \|_{\Hphi} \rightarrow 0$. First, fix a small $\epsilon>0$ and letting $\sigma$ denote normalized Lebesgue measure on $\T$ (or $\T^2$, depending on the context) choose a compact interval $\K \subseteq \mathbb{T}$ centered at $\tau_1$ such that $\sigma(\K) = \epsilon$. Then $\text{dist}(  \T \setminus \mathcal{K}, \tau_1) = \epsilon/2.$ Then by Lemma \ref{unitarylemma} and equations  \eqref{eqn:Cestimate}, \eqref{eqn:g1estimate}, we have
\[ 
\begin{aligned}
\| g _1^n\|^2_{\Hphi} & = \| g_1^n(\cdot, \tau_2) \|_{H^2}^2 \\
&= \int_{\K} |g^n_1(z_1, \tau_2)|^2d  \sigma(z_1) +  \int_{\T \setminus \K} \frac{|(1-\overline{\tau}_1z_1)g^n_1(z_1, \tau_2)|^2}{|z_1-\tau_1|^2} d  \sigma(z_1)  \\
&\le \sigma(\K) C^2 \|g^n_1\|^2_{\Hphi} + \frac{16C^2}{ \text{dist}(  \T \setminus \mathcal{K}, \tau_1)^2 }    \| (1-E_{\tau} D)g^n \|^2_{\Hphi}  \\
& \le \epsilon C^2 + \frac{64C^2}{\epsilon^2}  \| (1-E_{\tau} D)g^n \|^2_{\Hphi}. 
\end{aligned}
\]
Choose $N$ such that for all $n \ge N$, the latter term is less than $\epsilon.$ This shows $\| g _1^n\|_{\Hphi} \rightarrow 0.$ 

Now, consider $g_2^n$. By our original assumptions and the fact that  $\| g _1^n\|_{\Hphi} \rightarrow 0,$ we can conclude that 
\[ \|g_2^n\|_{\Hphi} \rightarrow 1 \ \text{ and } \|(1 - E_{ \tau }D)g^n_2 \|_{\Hphi} \rightarrow 0.\]
Examining the first component in the second limit yields
\[ \| \tau_1 (Dg_2^n)_1 \|_{\Hphi} \rightarrow 0, \text{ and so } \| M_{z_1}  (Dg_2^n)_1 \|_{H^2} \rightarrow 0\]
and similarly, examining the second component yields
\[ \| g_2^n - \tau_2 (Dg_2^n)_2\|_{\Hphi} \rightarrow 0, \text{ and so } \| M_{z_2} \left( \overline{\tau_2} g^n_2 - (Dg_2^n)_2\right) \| _{H^2} \rightarrow 0.\]
Thus \eqref{eqn:Dform} implies that 
\[ \| (1-M_{z_2}\overline{\tau_2})g^n_2 - g^n_2(0) \| _{H^2} = \| M_{z_1}  (Dg_2^n)_1  + M_{z_2}\left( (Dg_2^n - \overline{\tau_2} g^n_2)_2\right) \|_{H^2}  \rightarrow 0.\]
From earlier in the argument, we know that $|g^n_1(0) + g^n_2(0) | \rightarrow 0$ and $\| g^n_1\|_{H^2} \rightarrow 0$. This implies that  $g^n_2(0) \rightarrow 0$ and so,
\begin{equation} \label{eqn:g2Estimate} \| (1-M_{z_2}\overline{\tau_2})g^n_2\| _{H^2} \rightarrow 0.\end{equation}
 We claim that this implies $\|g^n_2\|_{\Hphi} \rightarrow 0.$ To see this, fix a small $\epsilon >0$ and $\K \subseteq \tilde{I}$ a compact interval centered at $\tau_2$ with $\sigma(\K) = \epsilon$. Note that such a $\K$ exists for $\epsilon$ sufficiently small. Then $\text{dist}(  \T \setminus \mathcal{K}, \tau_2) = \epsilon/2.$ Then by \eqref{eqn:Cestimate} and \eqref{eqn:g2Estimate},
 \[ 
\begin{aligned}
\| g _2^n\|^2_{\Hphi}
&= \int_{\T\times \K} |g^n_2(z)|^2d  \sigma(z)+ \int_{\T \times (\T \setminus \K)} \frac{|(1-\overline{\tau_2} z_2)g^n_2(z)|^2}{|z_2-\tau_2|^2} d  \sigma(z) \\
&\le \sigma(\K) C^2 \|g^n_2\|^2_{\Hphi} + \frac{1}{ \text{dist}(  \T \setminus \mathcal{K}, \tau_2)^2 } \| (1-z_2\overline{\tau_2})g^n_2\|^2 _{H^2} \\
& \le \epsilon C^2 + \frac{4}{\epsilon^2}  \| (1-M_{z_2}\overline{\tau_2})g^n_2\|^2_{H^2}.
\end{aligned}
\]
Choose $N$ such that for all $n \ge N$, the latter term is less than $\epsilon.$ This shows $\| g _2^n\|_{\Hphi} \rightarrow 0,$ a contradiction, which completes the proof.
\end{proof} 

It seems plausible that Theorem \ref{invertonI} should hold if $\phi$ is inner and extends continuously to $I_1\times I_2$ for open sets $I_1, I_2 \subseteq \mathbb{T}.$ While we have not been able to prove this, we can show that the $(\id-E_\tau D)$ operators have dense range.

\begin{proposition} \label{prop:range}
 Let $\phi$ be an inner function on $\mathbb{D}^2$ and assume $\phi$ extends continuously to $I_1\times I_2$ for open sets $I_1, I_2 \subseteq \mathbb{T}.$ Then, for each $\tau \in I_1 \times I_2$, the operator $(\id-E_\tau D)$ has dense range. 
\end{proposition}

\begin{proof} Fix $(\tau_1, \tau_2) \in I_1\times I_2$. As in the proof of Theorem \ref{invertonI}, assume that $g \in \mathcal{H}_{\phi}$ is orthogonal to the range of $\id-E_{\tau}D$. Write 
\[ g = \bbm g_1 \\ g_2 \ebm \ \text{ for $g_1 \in \mathcal{H}(K^{\min}_2),$ and $g_2 \in  \mathcal{H}(K^{\max}_1).$} \]
Then one can basically follow the proof of Theorem \ref{invertonI}, but directly apply Theorem $1.5$ in \cite{bk}, to conclude that 
\[ 
 (1 - \cc{\tau_1} w_1)g_1(w) = - (1 - \cc{\tau_2}w_2)g_2(w) \]
for all $w\in \Omega$, where $\Omega$ is an open set containing $\mathbb{D}^2 \cup (I_1 \times \mathbb{D}) \cup (\mathbb{D} \times I_2) \cup (I_1 \times I_2)$ and all elements of $\mathcal{H}_{\phi}$ extend to be holomorphic. on $\Omega$. For $w \in \Omega$ wherever the expression makes sense, define a function $f$ by 
\[ f(w) = \frac{g_1(w)}{1-\bar{\tau}_2 w_2} = \frac{ - g_2(w)}{1-\bar{\tau}_1 w_1}.\]
The first formula says $f$ is holomorphic on $\Omega \setminus \{ (w_1,w_2) : w_2 =\tau_2\}$,  and the second formula says $f$ is holomorphic on $\Omega \setminus \{ (w_1,w_2) : w_1 =\tau_1\}$. This implies that $f$ is holomorphic on $\Omega \setminus \{ \tau\}$. Holomorphic functions on open sets in $\mathbb{C}^2$ cannot have isolated singularities and so, $f$ must be holomorphic on $\Omega.$   To show that $f \in H^2(\mathbb{D}^2)$, choose compact sets $\K_1, \K_2 \subseteq \mathbb{T}$ containing $\tau_1$ and $\tau_2$ respectively such that $\K_1 \times \K_2 \subset I_1 \times I_2$ and  $\text{dist}(  \T \setminus \mathcal{K}_j, \tau_j)> 0$ for $j=1,2$. Then since $\K_1 \times \K_2  \subset \Omega$, $f$ is bounded on $\K_1 \times \K_2$ and we have: 
\[ 
\begin{aligned} 
\| f\|_{H^2}^2 &= \int_{\K_1 \times \K_2} |f(z)|^2  d \sigma(z)  + \int_{\T \times (\T \setminus \K_2)}  \left | \frac{g_1(z)}{1-\bar{\tau}_2 z_2}\right|^2  d \sigma(z)   \\
&+ \int_{(\T \setminus \K_1)\times \T} \left |\frac{g_2(z)}{1-\bar{\tau}_1 z_1} \right|^2 d\sigma(z)    < \infty.
\end{aligned}
\]
Furthermore, observe that for each $N \in \mathbb{N}$,
\[ f(w) =  \frac{g_1(w)}{1-\bar{\tau}_2 w_2}  = \sum_{n=0}^{N-1} g_1(w) \bar{\tau}^n_2 w^n_2 +  \bar{\tau}^N_2 w^N_2 \frac{g_1(w)}{1-\bar{\tau}_2 w_2} .\]
Since $g_1 \in \mathcal{H}(K^{\min}_2)$, we know $g_1 \perp_{H^2} w_2^n g_1$ for all $n >0$ and so the functions in the first sum are orthogonal to each other. This implies
\[ 
\begin{aligned} 
\| f\|_{H^2} &\ge \left \|  \sum_{n=0}^{N-1} g_1(w) \bar{\tau}^n_2 w^n_2 \right \|_{H^2}  -\left \| \bar{\tau}^N_2 w^N_2 \frac{g_1(w)}{1-\bar{\tau}_2 w_2} \right \|_{H^2} \\
&= \left( \sum_{n=0}^{N-1} \| w_2^n g_1(w) \|^2_{H^2}\right)^{1/2} - \| f\|_{H^2} \\
&= \sqrt{N} \| g_1 \|_{H^2} - \| f \|_{H^2}.
\end{aligned}
\]
Since this holds for all $N$, it follows that $g_1 \equiv 0$ and thus $g_2 \equiv 0$, which proves the claim.
\end{proof}

However, as discussed in the following remark, the proof showing that the $(\id-E_\tau D)$ operators are bounded below does not translate to this setting. 

\begin{rem} 
\label{rem:obstr} 
Numerous times, the proof of Theorem \ref{invertonI} uses the uniform boundedness of point evaluations delineated in \eqref{eqn:Cestimate}. For example, this is used to deduce that 
\begin{equation} \label{eqn:Goodbd} \int_{\T \times \K} |g^n_2(z)|^2d  \sigma(z) \lesssim \sigma(\K) \end{equation}
for $\K \subseteq I$ a small set containing $\tau_2$.

It is not clear how to obtain such bounds if $\phi$ only extends continuously to a more general product set $I_1 \times I_2$. To see how we might obtain this inequality using other means, recall that point evaluations on $\mathbb{D} \times I_2$ are bounded on $\mathcal{H}_{\phi}$. Then
\[  
\begin{aligned}
 \int_{\T \times \K} |g^n_2(z)|^2d  \sigma(z)   &= \int_{\T \times \K}  \lim_{r \nearrow 1} \left| \left \langle g^n_2, k^{\max}_{1,(rz_1, z_2)} \right \rangle_{H^2} \right|^2  d \sigma(z) \\
 & \le  \| g_2^n\|_{H^2}^2 
 \int_{\T \times \K}  \lim_{r \nearrow 1} \left \| k^{\max}_{1,(rz_1, z_2)} \right \|^2_{H^2} d \sigma(z). 
 \end{aligned}
 \] 
 Then \eqref{eqn:Goodbd} would follow if
 \[   \int_{\T \times \K}  \lim_{r \nearrow 1} \left \| k^{\max}_{1,(rz_1, z_2)} \right \|^2_{H^2} d \sigma(z) \lesssim \sigma(\K).\]
But, a straightforward computation using the model equation \eqref{modeleqn} and the one-variable Julia-Carath\'eodory theorem gives
\[  \lim_{r \nearrow 1} \left \| k^{\max}_{1,(rz_1, z_2)} \right \|^2_{H^2}  = 
 \lim_{r \nearrow 1} \frac{1-|\phi(rz_1, z_2)|^2}{1-r^2} \approx | \partial {z_1} \phi(z_1, z_2)|,\]
where $\partial {z_1} \phi(z_1,z_2)$ is the non-tangential derivative of $\phi(\cdot, z_2)$ at $z_1$, which is defined as long as $ \| k^{\max}_{1,(rz_1, z_2)} \|^2_{H^2}$ is bounded as $ r\nearrow 1$. Then the desired equality becomes 
\[   \int_{\T \times \K}  \lim_{r \nearrow 1} \left \| k^{\max}_{1,(rz_1, z_2)} \right \|^2_{H^2} d \sigma(z)  \approx 
 \int_{\T \times \K} |  \partial {z_1} \phi(z_1, z_2)| d \sigma(z) \lesssim \sigma(\K).\]
This uniform $H^1$ derivative bound certainly forces $\phi(\cdot, z_2)$ to be a finite Blaschke product for a.e.~$z_2 \in I_2$ and likely imposes even more stringent regularity conditions on $\phi$. Therefore, new techniques would be needed to show that  the $(\id-E_\tau D)$ operators are bounded below for $\phi$ that possess weaker regularity than that assumed in Theorem \ref{invertonI}.
\end{rem}

\section{Some algebraic identities} \label{algebra}
	In this section, we collect some algebraic identities satisfied by general noncommutative indeterminants which will allow us to convert between various
	representation formulae. 
	
	Algebraic realizations of noncommutative Schur-type functions are called Fornasini-Marchesini realizations, after the pioneering work in \cite{fm76}. We introduce a new algebraic version of the  usual noncommutative Herglotz realization (as in \cite{popescumajorant}), the so-called Herglotz-Nouveau formula. Finally, we refer to the noncommutative Nevanlinna realization \cite{ptd16, pascoeopsys, ptdroyal}.
	
	\subsection{The block $2$ by $2$ matrix inverse formula} \label{sec:inverseforms}
	In what follows, we will need formulas for inverses of block $2\times 2$ matrices with operator entries. For the ease of the reader, we include those here. Specifically, let $X$ be the following $2 \times 2$ block matrix with operator entries
\[X = \begin{bmatrix} Q & R \\
   S & V \end{bmatrix}.\]
Provided that certain related operators are invertible, this partition yields useful formulas for $X^{-1}$. For example, if $Q$ and $V -SQ^{-1}R$ are invertible, then so is $X$ and $X^{-1}$ is given by
   \[\begin{bmatrix} Q^{-1} +Q^{-1}R(V- SQ^{-1}R)^{-1} SQ^{-1} & -Q^{-1} R(V- SQ^{-1}R)^{-1} \\
   -(V- SQ^{-1}R)^{-1}SQ^{-1} & (V- SQ^{-1}R)^{-1} \end{bmatrix}.\]
Similarly, if $V$ and $Q -R V^{-1}S$ are invertible, then $X^{-1}$ exists and is given by the formula
\[ \begin{bmatrix} (Q -R V^{-1}S)^{-1}	 & - (Q -R V^{-1}S)^{-1}	 RV^{-1} \\
		-V^{-1}S( Q -R V^{-1}S)^{-1}	 & V^{-1} + V^{-1}S (Q -R V^{-1}S)^{-1}	 RV^{-1} \end{bmatrix}. \]
See for example, \cite[p. 18]{horjoh85}.		

	\subsection{Between Fornasini-Marchesini and Herglotz-Nouveau formulae}
		\begin{theorem}\label{SchurToHerglotz}
			Let $A, B, C, D, Z$ be operators taking various Hilbert spaces
			to other various Hilbert spaces such that the expression
			$$\Phi=A+B(1-ZD)^{-1}ZC$$ is well defined and
			$1-\Phi$ is invertible.
			Let
				$$\Theta = \frac{1+\Phi}{1-\Phi}.$$
			Then the expression $\left(\id - \bbm A & B \\ Z C & Z D \ebm \right)$ is an invertible operator and 
				\beq\label{newherg}
					\Theta = \bbm \id \\ 0 \ebm
					\ad  \left(\id - \bbm A & B \\ Z C & Z D \ebm \right)\inv \left(\id + \bbm A & B \\ Z C & Z D \ebm \right)  \bbm \id \\ 0 \ebm.
				\eeq		
		\end{theorem}
		\begin{proof}
			 By hypothesis, $1 - ZD$ is invertible and  $(1 - \Phi)\inv$, which is the Schur complement of $1 - ZD$ in $\bbm \id - A & - B \\ -ZC & 1 - ZD\ebm$, exists as $\id - \Phi$ is invertible by assumption. This implies that $\id - \bbm A & B \\ ZC & ZD\ebm$ is an invertible operator. 
			
			Applying the inverse formula for a block two by two matrix, we get
			\[
				\bbm \id - A & -B \\ -Z C & \id - Z D \ebm\inv = \bbm S & SB(\id - ZD)\inv \\ * & * \ebm
			\]
			where 
			\beq 
				S = (\id - A - B(\id-DZ)\inv Z C)\inv = (\id - \Phi)\inv,
			\eeq  
			and $*$ denotes some quantity that will be immaterial to our calculation.
			On substitution into \ref{newherg}, we get
			\begin{align*}
			&\bbm \id \\ 0 \ebm\ad  \left(\id - \bbm A & B \\ Z C & Z D \ebm \right)\inv \left(\id + \bbm A & B \\ Z C & Z D \ebm \right)  \bbm \id \\ 0 \ebm \\
			= &\bbm \id \\ 0 \ebm\ad  \bbm S & SB(\id - ZD)\inv \\ * & * \ebm \bbm A + \id & B \\ Z C & Z D + \id \ebm \bbm \id \\ 0 \ebm \\ 
			= &\bbm S & SB(\id - ZD)\inv\ebm \bbm \id + A \\ ZC \ebm \\
			= &S(\id + A) + SB(\id-Z D)\inv Z C \\
			= &S(\id + A + B(\id -Z D)\inv Z C) \\
			= &\frac{\id + \Phi}{\id - \Phi},
			\end{align*}
			which proves the claim.
		\end{proof}

	\subsection{Between Herglotz-Nouveau and Nevanlinna formulae}
		\begin{theorem} \label{thm:Nform}
			Let 
				$U = \begin{bmatrix} A & B \\ C & D \end{bmatrix}$
			be a block operator such that $\id-U$ is invertible. Let $Z$ be an operator such that $\id-Z$ is invertible and 
			
				\beq
					\Theta = \bbm \id \\ 0 \ebm
					\ad  \left(\id - \bbm A & B \\ Z C & Z D \ebm \right)\inv \left(\id + \bbm A & B \\ Z C & Z D \ebm \right) \bbm \id \\ 0 \ebm
				\eeq
				is well defined.
			Let $W=i\frac{\id+Z}{\id-Z}$ so $Z = \frac{W-i}{W+i}$,
			let $$T = i (\id +U)(\id-U)^{-1} = \begin{bmatrix} T_{11} & T_{12}  \\ T_{21} & T_{22} \end{bmatrix},$$
			and let
				$\Psi = i\Theta.$			
			Then the expression $(W + T_{22})$ is an invertible operator and
				\[ \Psi =  T_{11} -T_{12} (W +T_{22})^{-1}T_{21}.\]
		\end{theorem}

		\begin{proof} For ease of notation, set 
			$ \alpha = \begin{bmatrix} \id \\ 0 \end{bmatrix}.$ The formula for $\Theta$ gives
		\[ 
		\begin{aligned} 
		\Psi &= i \Theta = i \alpha^* \left( \id - \begin{bmatrix} \id & 0 \\
		0 &  \frac{W-i}{W+i} \end{bmatrix} \begin{bmatrix} A & B \\ C & D \end{bmatrix} \right)^{-1} \left( \id + \begin{bmatrix} \id & 0 \\
		0 &  \frac{W-i}{W+i} \end{bmatrix}  \begin{bmatrix} A & B \\ C & D \end{bmatrix} \right) \alpha \\
		& = i\alpha^* \left(  \begin{bmatrix} \id & 0 \\  0 & W+i \end{bmatrix}  - \begin{bmatrix} \id & 0 \\
		0 &  W-i \end{bmatrix} U \right)^{-1} \left( \begin{bmatrix} \id & 0 \\  0 & W+i \end{bmatrix} + \begin{bmatrix} \id & 0 \\
		0 &  W-i \end{bmatrix}  U \right) \alpha \\
		& =i \alpha^* \left(  \begin{bmatrix} \id & 0 \\  0 & W \end{bmatrix}(\id-U)  + \begin{bmatrix} 0 & 0 \\
		0 &  i \end{bmatrix} (\id+U)  \right)^{-1} \left( \begin{bmatrix} \id & 0 \\  0 & W \end{bmatrix}(\id+U)  + \begin{bmatrix} 0 & 0 \\
		0 &  i \end{bmatrix} (\id-U)\right) \alpha. 
		\end{aligned} \] 
		Recalling that $T = i (\id +U)(\id-U)^{-1}$, we have
		\[ 
		\begin{aligned} 
		\Psi = \alpha^* (\id-U)^{-1}  \left(  \begin{bmatrix} \id & 0 \\  0 & W \end{bmatrix}  + \begin{bmatrix} 0 & 0 \\
		0 &  \id \end{bmatrix} T  \right)^{-1} \left( \begin{bmatrix} \id & 0 \\  0 & W \end{bmatrix}T  + \begin{bmatrix} 0 & 0 \\
		0 &  -\id \end{bmatrix}\right) (\id-U) \alpha. 
		\end{aligned} \] 
		The expression $\left(\bbm \id & 0 \\ 0 & W \ebm + \bbm 0 & 0 \\ 0 & \id\ebm T\right)\inv$ is the conjugation of a well defined expression from the original equation by invertible operators, and thus remains well defined. Then note that
		\[
		\left(\bbm \id & 0 \\ 0 & W \ebm + \bbm 0 & 0 \\ 0 & \id\ebm T\right)\inv = \bbm \id & 0 \\ T_{21} & W + T_{22}\ebm\inv
		\]
		which implies that the expression $W + T_{22}$ is invertible.
		
		Now, writing 
		\[\begin{bmatrix} \id & 0 \\  0 & W \end{bmatrix}T  + \begin{bmatrix} 0 & 0 \\
		0 &  -\id \end{bmatrix} =  \left(  \begin{bmatrix} \id & 0 \\  0 & W \end{bmatrix} + \begin{bmatrix} 0 & 0 \\
		0 &  \id \end{bmatrix} T\right)T -   \begin{bmatrix} 0 & 0 \\  0 & \id \end{bmatrix} (\id +T^2),\] gives
		\[ 
		\begin{aligned}
		\Psi  &=  \alpha^* (\id-U)^{-1} T (\id-U) \alpha  \\&\hspace{.5in}-\alpha^* (\id-U)^{-1}  \left(  \begin{bmatrix} \id & 0 \\  0 & W \end{bmatrix}  + \begin{bmatrix} 0 & 0 \\
		0 &  \id \end{bmatrix} T  \right)^{-1}   \begin{bmatrix} 0 & 0 \\  0 & \id \end{bmatrix} (\id +T^2) (\id-U) \alpha \\
		& = \alpha^* T  \alpha  - \alpha^* (\id-U)^{-1}  \left(  \begin{bmatrix} \id & 0 \\  0 & W \end{bmatrix}  + \begin{bmatrix} 0 & 0 \\
		0 &  \id \end{bmatrix} T  \right)^{-1}   \begin{bmatrix} 0 & 0 \\  0 & \id \end{bmatrix} (\id +T^2) (\id-U) \alpha.  
		\end{aligned}
		\]
		Observe that 
		\[  (\id +T^2) (\id-U)  = ((\id-U)^2 - (\id +U)^2)(\id-U)^{-1}  = - 4 U (\id-U)^{-1}, \]
		and by the inversion formula for $2\times 2$ block operators, since $\id$ and $T_{22} +W$ are both invertible, we have
		\[   
		\begin{aligned} 
		\left(  \begin{bmatrix} \id & 0 \\  0 & W \end{bmatrix}  + \begin{bmatrix} 0 & 0 \\
		0 &  \id \end{bmatrix} T  \right)^{-1}  \begin{bmatrix} 0 & 0 \\  0 & \id \end{bmatrix} &= \begin{bmatrix} \id & 0 \\
		T_{21} & W + T_{22} \end{bmatrix}^{-1} \begin{bmatrix} 0 & 0 \\  0 & \id \end{bmatrix}  \\
		& = 
		\begin{bmatrix} \id & 0 \\  -(W+T_{22})^{-1}T_{21} & (W+T_{22})^{-1}  \end{bmatrix}  \begin{bmatrix} 0 & 0 \\  0 & \id \end{bmatrix}    \\
		& = \begin{bmatrix}  0 & 0 \\
		0 & (W +T_{22})^{-1} \end{bmatrix} = \begin{bmatrix} 0 \\ \id \end{bmatrix} (W +T_{22})^{-1} \begin{bmatrix} 0 & \id \end{bmatrix}.
		\end{aligned}
		\]
		Then, using the definition of $\alpha$, the equation simplifies to
		\[ \Psi = T_{11} +4  \begin{bmatrix} \id & 0 \end{bmatrix} (\id-U)^{-1} \begin{bmatrix} 0 \\ \id \end{bmatrix} (W +T_{22})^{-1} \begin{bmatrix} 0 & \id \end{bmatrix} U (\id-U)^{-1} \begin{bmatrix} \id \\ 0 \end{bmatrix} .\]
		Further, observe that 
		\[ 
		\begin{aligned}
		T&= i( \id-U +2U)(\id-U)^{-1}= i +2iU(\id-U)^{-1}, \ \text{ so }  \ 2 U(\id-U)^{-1} = -iT -\id, \\
		T&= i(U -\id +2)(\id-U)^{-1}= -i +2i(\id-U)^{-1}, \  \text{ so} \ 2 (\id-U)^{-1} = -iT +\id.
		\end{aligned}
		\]
		Those formulas imply that 
		\[ 2\begin{bmatrix} \id & 0 \end{bmatrix} (\id-U)^{-1} \begin{bmatrix} 0 \\ \id \end{bmatrix}  = -iT_{12} \ \text{ and } \ 2\begin{bmatrix} 0 & \id \end{bmatrix} U (\id-U)^{-1} \begin{bmatrix} \id \\ 0 \end{bmatrix} = -iT_{21}\]
		and so, the formula for $\Psi$ becomes
		\[ \Psi =  T_{11} -T_{12} (W +T_{22})^{-1}T_{21},\]
		which is what we were trying to show.
		\end{proof}

\section{Concrete Nevanlinna formulae} \label{concretenp}

We can use Theorem \ref{thm:Nform} to translate the concrete realizations for inner functions on $\mathbb{D}^2$ from Section \ref{bkyourway} to realizations for Pick functions on $\Pi^2$. 

First, assume that $\phi$ is an inner function on $\mathbb{D}^2$. 
Then $\phi$ has a model representation as in \eqref{eqn:AD1}, where the realization operator
\[
U := \bbm A & B \\ C & D \ebm : \bbm \C \\ \hilbert_\phi \ebm \mapsto \bbm \C \\ \hilbert_\phi \ebm
\]
has block formulas given in Theorem \ref{thm:AD}. Define 
\[ \alpha: \mathbb{D}\rightarrow\Pi \text{ by } \alpha(z) = i\frac{1+z}{1-z} \text{ and }  \alpha^{-1}:\Pi \rightarrow \mathbb{D} \text{ by } \alpha^{-1}(w) = \frac{ w-i}{w+i}.\]
Then $\psi= \alpha \circ \phi \circ \alpha\inv$ is an inner Pick function on $\Pi^2$, where inner means that $\psi$ is real valued for almost every point in $\R^2$. Here, we should mention that the notation $ \phi \circ \alpha\inv$ is short-hand for $ \phi \circ (\alpha\inv,\alpha\inv)$ and will be used throughout the rest of the paper.

If $(\id-U)$ is invertible, then Theorem \ref{thm:Nform} implies that if
 $$T = i (\id +U)(\id-U)^{-1} = \begin{bmatrix} T_{11} & T_{12}  \\ T_{21} & T_{22} \end{bmatrix}$$
on $\mathbb{C} \oplus \mathcal{H}_{\phi}$, then 		
\[ \psi (w) =  T_{11} -T_{12} (E_w +T_{22})^{-1}T_{21}\]
for $w \in \Pi^2$. If $\phi$ has sufficient regularity at $(1,1)$, then one can deduce the following explicit formulas for $T_{11}$, $T_{12}$, and $T_{21}$. Progress towards a similar description of $T_{22}$ is described in Remark \ref{rem:T22}.

\begin{theorem} \label{thm:nevanlinna} Let $\phi$ and $\psi$ be as above. Assume that $\phi$ extends continuously to a neighborhood of $(1,1)$ on $\mathbb{T}^2$ with $\phi(1,1) \ne 1$ and \eqref{eqn:AD1} extends to $(1,1)$. Then by Theorem 1.5 in \cite{bk}, $f(1,1)$ exists for all $f \in \mathcal{H}_{\phi}$ and there is a $k_1\in \mathcal{H}_{\phi}$ such that $f(1,1) = \langle f, k_1\rangle_{\mathcal{H}_{\phi}}$ for all $f \in \hilbert_\phi$. Then:
\begin{itemize}
\item[i.] For all $x \in \mathbb{C}$, $T_{11}$ is given by
\[ T_{11} x = i\frac{1+\phi(1,1)}{1-\phi(1,1)} x. \]
\item[ii.] For all $x \in \mathbb{C}$, $T_{21}$ is given by
\[ T_{21} x = \frac{2i \phi(1,1)}{1-\phi(1,1)} k_1x.\]
\item[iii.] For all $f  \in \mathcal{H}_{\phi}$, $T_{12}$ is given by
\[ T_{12} f = \frac{2 i}{1-\phi(1,1)} f(1,1).\]
\end{itemize}
\end{theorem}

It follows from the proof that weaker regularity conditions are needed to obtain the formulas for $T_{11}$ and $T_{21}.$

\begin{proof} To obtain formulas for $T$, we require a formula for $(1-U)^{-1}$. Let $c_1:=(1-\phi(1,1))^{-1}.$ Using the second formula in Section \ref{sec:inverseforms} and the fact that $\phi(1,1) = A+B(\id-D)^{-1} C$, one can obtain
\[ (1-U)^{-1} = \begin{bmatrix}  c_1 & c_1 B(1-D)^{-1} \\
c_1(1-D)^{-1} C & (1-D)^{-1} + c_1(1-D)^{-1}CB(1-D)^{-1}  \end{bmatrix}. \]
Using $T = i(1+U)(1-U)^{-1}$, block matrix multiplication, and straightforward simplification gives
\[ \begin{aligned}
T_{11} &= i c_1(1 +A + B(1-D)^{-1} C )= i \tfrac{ 1 + \phi(1,1)}{1-\phi(1,1)}; \\
T_{21} & = ic_1 (C + (1+D)(1-D)^{-1} C) = \tfrac{2i}{1-\phi(1,1)}(1-D)^{-1} C; \\
T_{12} &= i (1+A)c_1B(1-D)^{-1} +i B(1-D)^{-1} + ic_1B(1-D)^{-1} CB(1-D)^{-1} \\
&=\tfrac{2i}{1-\phi(1,1)} B(1-D)^{-1}; \\
T_{22} &= i( 1+D)(1-D)^{-1} + c_1i\left( 1 + (1+D)(1-D)^{-1} \right) CB(1-D)^{-1} \\
&=  i( 1+D)(1-D)^{-1} +\tfrac{2i}{1-\phi(1,1)} (1-D)^{-1} CB (1-D)^{-1}.
\end{aligned}
\]
The formula for $T_{11}$ is immediate. To obtain the formula for $T_{21}$, set 
\[ k_w = k_{2,w}^{\min} + k_{1,w}^{\max}.\]
Then by Proposition \ref{Dshift}, 
\begin{equation} \label{eqn:Dkw} Dk_w = E_{\bar{w}} k_w -\overline{ \phi(w) } F,\end{equation}
where $F$ is a function defined in Proposition \ref{Dshift} and by Theorem \ref{thm:AD}, $Cx = Fx$ for all $x \in \mathbb{C}$. By Theorem 1.5 in \cite{bk}, as $w\rightarrow (1,1)$ with $w \in \mathbb{D}^2$, $k_w \rightarrow k_1$ weakly in $\mathcal{H}_\phi$. One can use this to show
\[ D k_1 = k_1 -\overline{\phi(1,1)} F, \ \text{ which implies } \ (1-D)^{-1} F = \phi(1,1) k_1.\]
It follows immediately that for $x \in \mathbb{C}$, 
\[ T_{21} x = 2ic_1(1-D)^{-1} C x = 2ic_1(1-D)^{-1} F x =\tfrac{2i \phi(1,1)}{1-\phi(1,1)}k_1 x.\]
To study $T_{12}$ and $T_{22}$, recall that we assumed $\phi$ continuously extends to a neighborhood of $(1,1)$. This implies that linear combinations of the functions  $(1-E_{\bar{w}})k_w$ are dense in $\mathcal{H}_{ \phi}$. To see this, assume that $g = g_1 + g_2 \in \mathcal{H}_{\phi}$ and for all $w \in \mathbb{D}^2$,
\[ 0=\langle g, (1-E_{\bar{w}})k_w \rangle_{\mathcal{H}_{\phi}} = (1-w_1) g_1(w) + (1-w_2)g_2(w).\]
Then the arguments in the proof of Proposition \ref{prop:range} imply that $g \equiv 0.$ Thus it suffices to find a linear formula for $T_{12}$ on functions of the form $(1-E_{\bar{w}})k_w$.

To that end, note that \eqref{eqn:Dkw} implies that 
\[ (1-E_{\bar{w}}) k_w = k_w - D k_w - \overline{\phi(w)} F.\]
Then by Theorem \ref{thm:AD} and the formula \eqref{modeleqn},
\[
\begin{aligned}
 T_{12} (1-E_{\bar{w}}) k_w  &= \tfrac{2i}{1-\phi(1,1)} B (1-D)^{-1} \left( k_w - D k_w - \overline{\phi(w)} F \right)  \\
 & = \tfrac{2i}{1-\phi(1,1)} B (k_w -\overline{\phi(w)} \phi(1,1) k_1) \\
 & = \tfrac{2i}{1-\phi(1,1)}\left( k_w(0)- \overline{\phi(w)}\phi(1,1) k_1(0)\right) \\
 & = \tfrac{2i}{1-\phi(1,1)}\left( (1-\overline{\phi(w)}\phi(0,0)) - \overline{\phi(w)}\phi(1,1) (1-\overline{\phi(1,1)}\phi(0,0))\right) \\
 & =\tfrac{2i}{1-\phi(1,1)}\left( 1-\overline{\phi(w)}\phi(1,1)\right) \\
 & = \tfrac{2i}{1-\phi(1,1)}(1-E_{\bar{w}}) k_w(1,1),
 \end{aligned}
 \] 
 which establishes the formula for $T_{12}$ and completes the proof of this theorem. Partial results concerning $T_{22}$ are given in Remark \ref{rem:T22}. 
\end{proof}

\begin{rem}\label{rem:T22} Recall from Theorem \ref{thm:nevanlinna} that 
\[ T_{22} =  i( 1+D)(1-D)^{-1} +\tfrac{2i}{1-\phi(1,1)} (1-D)^{-1} CB (1-D)^{-1}.\]
The second piece of $T_{22}$ combines the operators seen in $T_{21}$ and $T_{12}$. Thus, we can combine our formulas for those two operators as follows: for all $f \in \mathcal{H}_{\phi},$
\[ \tfrac{2i}{1-\phi(1,1)}  (1-D)^{-1} CB (1-D)^{-1} f =\tfrac{2i}{1-\phi(1,1)}   (1-D)^{-1} C f(1,1) = \tfrac{2i f(1,1)}{1-\phi(1,1)} k_1.\]
We have not been able to deduce an explicit formula for the first piece of $T_{22}$. However, we can compute it on the functions $(1-E_{\bar{w}})k_w$ and so, under the regularity assumptions of Theorem \ref{thm:nevanlinna}, know its behavior on a dense set in $\mathcal{H}_{\phi}$. Specifically, using Proposition \ref{Dshift},
\[ 
\begin{aligned}
 i( 1&+D)(1-D)^{-1} (1-E_{\bar{w}})k_w 
=  i( 1+D)(k_w -\overline{\phi(w)} \phi(1,1) k_1)  \\
& = i(k_w -\overline{\phi(w)}\phi(1,1) k_1 + E_{\bar{w}}k_w - \overline{\phi(w)}F -\overline{\phi(w)} \phi(1,1)(k_1 - \overline{\phi(1,1)} F )) \\
& = i (1 + E_{\bar{w}})k_w -2i\overline{\phi(w)}\phi(1,1) k_1.
\end{aligned}
\]
We have not been able to find a  bounded linear operator on $\mathcal{H}_{\phi}$ that gives this formula on the functions $(1-E_{\bar{w}})k_w$ and leave that as an open question.
\end{rem}
\section{The McCarthy Champagne conjecture} \label{mccarthy}
A large motivation for our present developments is the McCarthy Champagne Conjecture (MCC). A comprehensive discussion of the MCC was already provided in the introduction, but for the ease of the reader, let us recall the statement of the MCC here:

\begin{center} \textbf{(MCC): Every $d$-variable Pick-Agler function that analytically continues across an open convex set $E \subseteq \mathbb{R}^d$ is globally matrix monotone when restricted to $E$. } \end{center} 

Equivalently, this says that every locally matrix monotone function on an open convex set $E \subseteq \mathbb{R}^d$ is globally matrix monotone on $E$. In this section, we establish the MCC for two-variable Pick functions arising from quasi-rational functions and for $d$-variable perspective functions.

	\subsection{Quasi-rational functions} As in Section \ref{concretenp}, let 
		 $\alpha: \mathbb{D} \rightarrow \Pi$ denote the Cayley transform given by
		$\alpha(z) = i\left( \tfrac{1+z}{1-z} \right).$ Then  as a direct result of Theorem \ref{invertonI} combined with Theorem \ref{thm:Nform}, we can establish the MCC for Pick functions arising from quasi-rational functions.

		\begin{theorem}\label{mccarthy1} Let $I \subseteq \mathbb{T}$ be open,
		let  $\phi$ be a nonconstant two-variable quasi-rational function on $\mathbb{T} \times I$, and define a Pick function $f$ by $f = \alpha \circ \phi \circ
		\alpha^{-1}$. Then $f$ is globally matrix monotone on every open rectangle $E \subseteq \mathbb{R} \times \alpha(I)$ such that $\phi$ does not attain the value $1$ on $\alpha^{-1}(E).$ 
		\end{theorem}
		\begin{proof} Let $E' = J_1 \times J_2$ be a finite open rectangle with $\overline{E'} \subseteq E$. 
		Since $E'$ is arbitrary, it suffices to show $f$ is globally matrix monotone on $E'$. Let $\beta = (\beta_1, \beta_2)$ be a pair of conformal self
		maps of $\Pi$ such that 
		\[ \beta(E') \subseteq (0,\infty)^2 \text{ and }  (0,\infty)^2 \cup (\infty, \infty) \subseteq \beta (E)\]
		and  furthermore $f(\beta^{-1}(\infty, \infty)) \in \mathbb{R}.$ Define $F = f \circ \beta^{-1}$. Observe that each $\beta_j$ is a one variable matrix monotone function. Thus, to show  $f$ is globally matrix monotone on $E'$, we need only show that $F$ is globally matrix monotone on $(0,\infty)^2$. 
		
 To that end, observe that $F = \alpha \circ \Phi \circ \alpha^{-1}$, where $\Phi = \phi \circ \gamma$ and $\gamma = (\gamma_1, \gamma_2)$ is a pair of conformal self maps of $\mathbb{D}$ defined by $\gamma_j = \alpha^{-1} \circ \beta_j^{-1} \circ \alpha$. Then $\Phi$ is quasi-rational on $\mathbb{T} \times I'$, where $I' = \gamma_2^{-1}(I)$.  Tracing through the assumptions about $\phi$, $E$, and $\beta$ shows that $(1,1) \in \mathbb{T} \times I'$, $\Phi(1,1) \ne 1$, the set $\alpha^{-1}((0,\infty)^2) \subseteq \mathbb{T} \times I',$ and  $\Phi$ does not attain the value $1$ on $\alpha^{-1}((0,\infty)^2).$
		
Let $U=V^*$ be the coisometry from Theorem \ref{thm:AD} associated to $\Phi$ and defined in \eqref{eqn:Udef}. Then 
\begin{equation} \label{eqn:Phi} \Phi(z) = A + B(\id - E_z D)\inv E_z C
 \end{equation}
for $z \in \mathbb{D}^2$ and by Theorem \ref{invertonI}, $(\id - E_\tau D)\inv$ exists for all $\tau \in \mathbb{T} \times I'$. This implies that \eqref{eqn:Phi} extends to all $\tau \in  \mathbb{T} \times I'$, including $(1,1)$.
As $\Phi(1,1)\ne1$ and $(1-D)^{-1}$ exists, standard information about inverses for block $2\times 2$ operators, see Section \ref{sec:inverseforms}, implies that $\id-U$ is invertible. Since $U$ is a co-isometry and $1-U$ is invertible, the von Neumann-Wold decomposition implies that $U$ is unitary. 

Fix any $w \in \Pi^2 \cup (0, \infty)^2$, so that $w = \alpha(z)$ for some $z \in \mathbb{D}^2 \cup \alpha^{-1}((0,\infty)^2)$. Then we can apply Theorems \ref{SchurToHerglotz} and \ref{thm:Nform} with $Z = E_z$ and $W=E_w$ to conclude that 
\begin{equation} \label{eqn:FPF}  F(w) = T_{11} - T_{12}(T_{22}+E_w)^{-1}T_{21},\end{equation}
where $T=i(1+U)(1-U)^{-1}$. \color{black} Since $U$ is unitary, $T$ is self-adjoint and since $(T_{22}+E_w)^{-1}$ exists for $w\in (0, \infty)^2$, $T_{22}$ must be positive semidefinite. Observe that \eqref{eqn:FPF} has a natural extension to a map sending all pairs of matrix inputs with positive imaginary part to outputs with positive imaginary part, see for example \cite[Theorem 5.7]{pastd13}. Since $T_{22}$ is positive semidefinite, \eqref{eqn:FPF} extends further to all pairs of positive matrices as inputs for $w_1,w_2.$ As the cone of pairs of positive matrices is a free, convex set, the noncommutative L\"owner theorem, see \cite[Theorem 1.7]{pastd13} as well as \cite{pascoeopsys, palfia, ptdroyal}, implies that $F$ is globally matrix monotone on $(0,\infty)^2.$
\end{proof}

	\subsection{Perspective functions}
	
	Define a \dfn{commutative perspective function} $f$ to be a locally matrix monotone function on an open cone $C\subseteq (0,\infty)^d$
	such that $f(tz)=tf(z)$ when $t\in\mathbb{R}^+.$ {Perspective functions appear in the work of And\^o and Kubo in the context of monotone functions and operator means via L\"owner's theorem, and in the convex optimization regime in a series of papers by Effros, Hansen, and others. In particular, Effros and Hansen prove that convex non-commutative perspectives arise from convex commutative perspectives. See, e.g. \cite{andokubo80, effros09, ebadian2011, effroshansen13}.} 
	\begin{theorem}\label{mccarthy2}
		If $f$ is a commutative perspective function on an open cone $C\subseteq (0,\infty)^d$, then $f$ is globally matrix monotone on $(0, \infty)^d$.
	\end{theorem}
	\begin{proof}

	By the Theorem \ref{commlow} (the commutative L\"owner theorem), $f$ has an analytic continuation as a Pick-Agler function $f$ on the poly upper half plane $\Pi^d$. Since $C\subseteq \mathbb{R}^d$ is open, the identity theorem implies that this analytic continuation is unique on $\Pi^d$.  For any $t \in \R^+$, consider the Pick-Agler function $g(z) = t f(z/t)$. Because $f$ is positively homogenous on $C$,  $f=g$  on $C$ and by the uniqueness of the extension, $f=g$ on $\Pi^d$. Thus, $f$ is positively homogeneous on $\Pi^d$, which  immediately implies that the non-tangential value of $f$ at $0$ is $0$.

Now we show that $f$ has a useful Nevanlinna representation. To do so, we need to show that $f$ is sufficiently well behaved at $0$ (that is, $f$ has a carapoint at $0$ in the language of \cite{aty13}). Let $H(z) =f(-1/z)$. Then
	\begin{align*}
	&\phantom{=} \liminf_{y \to \infty} y \abs{H(iy, \ldots, iy)} \\
	&= \liminf_{y\to\infty} y \abs{f(i \tfrac{1}{y}, \ldots, i \tfrac{1}{y})} \\
	&= \liminf_{y \to \infty} \abs{f(i, \ldots, i)} \\
	&< \infty.
	\end{align*}
Given this, Theorem 1.6 in \cite{aty13} says that there must exist a Hilbert space $\hilbert$, a densely-defined self-adjoint operator $A$ on $\hilbert$, positive semidefinite contractions $Y_1, \dots, Y_d$ summing to $\id$ on $\mathcal{H}$, and a vector $\nu \in \mathcal{H}$ so that for all $z \in \Pi^d$, 
	\[
	H(z) = \ip{(A - \sum z_i Y_i)\inv \nu}{\nu}_{\hilbert}.
	\]
	Therefore, the same objects give a representation of $f$ by
	\[
	f(z) = \ip{(A + \sum z_i\inv Y_i)\inv \nu}{\nu}_{\hilbert}
	\]
	for all $z \in \Pi^d$.
	So, since $f(z) = tf(z/t)$,
		$$f(z)=\langle (A+\sum z_i^{-1}Y_i)^{-1}\nu,\nu \rangle_{\hilbert} = \langle (\tfrac{1}{t}A+\sum z_i^{-1}Y_i)^{-1}\nu,\nu\rangle_{\hilbert}.$$
By letting $t\rightarrow \infty,$  we can assume $A=0$. Then since $(\sum z_i^{-1}Y_i)^{-1}$ is well defined for all $d$ tuples of positive matrices, the noncommutative L\"owner theorem \cite{pastd13} implies that $f$ is globally matrix monotone on $(0, \infty)^d.$
	\end{proof}


\bibliography{references}
\bibliographystyle{plain}

\end{document}